\documentclass[hidelinks]{article}
\usepackage{graphicx} 
\usepackage[a4paper, margin=2.5cm]{geometry}
\usepackage{hyperref}
\usepackage{amsthm}
\usepackage{thm-restate}
\usepackage{geometry}
\usepackage{amssymb}
\usepackage{amsmath}
\usepackage{cleveref}
\usepackage{comment}
\usepackage{xcolor}
\usepackage[shortlabels]{enumitem}
\usepackage{listings}
\usepackage{blkarray}
\usepackage{lmodern}
\usepackage[english]{babel}
\usepackage{float}
\usepackage{tikz}

\newtheorem{theorem}{Theorem}[section]
\newtheorem{lemma}[theorem]{Lemma}
\newtheorem{corollary}[theorem]{Corollary}

\newtheorem{proposition}[theorem]{Proposition}
\newtheorem{conjecture}[theorem]{Conjecture}

\theoremstyle{definition}
\newtheorem{definition}[theorem]{Definition}

\newcommand{\cC}{\mathcal{C}}
\newcommand{\cH}{\mathcal{H}}

\newcommand{\cS}{\mathcal{S}}

\newcommand{\eps}{\varepsilon}

\newcommand{\vv}{\mathbf{v}}
\newcommand{\vu}{\mathbf{u}}
\newcommand{\vw}{\mathbf{w}}

\title{Universality for graphs with bounded density}

\author{
Noga Alon\thanks{Department of Mathematics, Princeton University, USA.
			Email: \texttt{nalon@math.princeton.edu}. This author was supported by the NSF under grant DMS-2154082.}
\and
Natalie Dodson\thanks{Middlebury College, USA.
                    Email: \texttt{ndodson@middlebury.edu}. This author was supported by the NSF under grant DMS-2150434.}
\and
Carmen Jackson\thanks{Northwestern University, USA.
                      Email: \texttt{carmenjackson2024@u.northwestern.edu}. This author was supported by the NSF under grant DMS-2150434.}
\and
Rose McCarty\thanks{Department of Mathematics, Princeton University, USA. Current affiliation: School of Mathematics and School of Computer Science, Georgia Institute of Technology, USA.
		Email: \texttt{rmccarty3@gatech.edu}. This author was supported by the NSF under grants DMS-2150434 and DMS-2202961.}
\and
Rajko Nenadov\thanks{School of Computer Science, University of Auckland, New Zealand. Email: \texttt{rajko.nenadov@auckland.ac.nz}}
\and
Lani Southern\thanks{Willamette University, USA.
Email: \texttt{lmsouthern@willamette.edu}. This author was supported by the NSF under grant DMS-2150434.}
}
\date{}

\begin{document}

\maketitle

\begin{abstract}
A graph $G$ is \emph{universal} for a (finite) family $\cH$ of graphs if every
$H \in \cH$ is a subgraph of $G$. For a given family $\cH$, the goal is to determine  the smallest number of edges an $\cH$-universal graph can have. With the aim
of unifying a number of recent results, we consider a family of graphs
with bounded density. In particular, we construct a graph with 
$$
    O_d\left( n^{2 - 1/(\lceil d \rceil + 1)} \right)
$$
edges which contains every $n$-vertex graph with density at most $d \in
\mathbb{Q}$ ($d \ge 1$), which is close to a lower bound $\Omega(n^{2 - 1/d -
o(1)})$ obtained by counting lifts of a carefully chosen (small) graph. When restricting the maximum degree of such graphs to be constant, we obtain a near-optimal universality. If we
further assume $d \in \mathbb{N}$, we get an asymptotically optimal construction.
\end{abstract}

\section{Introduction}

A graph $G$ is \emph{universal} for a (finite) family $\cH$ of graphs if every
$H \in \cH$ is a (not necessarily induced) subgraph of $G$. The complete graph with $n$ vertices is universal for the family of all graphs with $n$ vertices, and this is clearly the smallest universal graph for this family. However, if we restrict our attention to a family of graphs with some additional properties, more efficient (in terms of the number of edges) universal graphs might exist. This is a natural combinatorial question, with applications in VLSI circuit design \cite{bhatt84vlsi}, data storage~\cite{chung83storage}, and simulation of parallel computer architecture \cite{bhatt1986optimal}.

The problem of estimating the minimum possible number of edges in a universal graph for
various families has received a considerable amount of
attention. The previous work deals with
families of graphs with properties which naturally bound their density, such as graphs with bounded maximum degree~\cite{alon02sparse,alon08optimal,alon07sparsebounded,alon2000universality,alon2001near}, forests \cite{chung78tree,chung83spanning,chung78tree2,friedman87tree} and, more generally, graphs with bounded degeneracy~\cite{allen2023universality,nenadov2016ramsey}, as well as families of graphs with additional structural properties such as planar graphs \cite{babai82planar,esperet23planar} and graphs with small separators \cite{capalbo02planar,capalbo99small,chung90separator}, to name a few. Our focus is on the former case. Aiming to unify these results, we initiate the study of universality for a family of graphs with bounded density and no other assumptions. The density of a graph $H$ is defined as
$$
    m(H) = \max_{H' \subseteq H} \frac{e(H')}{v(H')}
$$
where $e(H')$ is the numbers of edges of $H'$ and $v(H')$ is the number of its vertices. 
In plain words, a graph $G$ has density $d \in \mathbb{Q}$ if not only the
number of edges of $G$ is at most $v(G) d$, but this also holds for every
subgraph of $G$. For $d \in \mathbb{Q}$ and $n \in \mathbb{N}$, we denote with
$\cH_d(n)$ the family of all graphs with $n$ vertices and density at most $d$.

As already hinted, the family of graphs with bounded density generalizes many interesting families.
For example, graphs with maximum degree $d$ have density at most
$d/2$. Forests have density arbitrarily close to $1$, and $d$-degenerate
graphs  and, more generally, graphs of arboricity $d$, have density
arbitrarily close to $d$ (a graph is $d$-degenerate if every subgraph has
minimum degree at most $d$). Note that every graph of density at most $d$ is also $\lfloor 2d \rfloor$-degenerate, thus bounded density implies bounded degeneracy. However, as we are aiming for an optimal dependence on the parameters, this implication does not suffice. A number of (almost-)optimal results have been obtained in some of these cases~\cite{allen2023universality,alon08optimal,chung83spanning}, and we conjecture that for all such families the bound on the size of a smallest universal graph is largely governed by the density. Therefore, generalizing all of these results, we believe the following is true.

\begin{conjecture} \label{conj:bounded_density}
    For every $d \in \mathbb{Q}$, $d > 1$, and $n \in \mathbb{N}$, there exists
    a graph with $G$ with
    $$
        e(G) \le C n^{2 - 1/d}
    $$
    edges which is $\cH_{d}(n)$-universal, where $C = C(d)$.
\end{conjecture}

If true, the bound in Conjecture \ref{conj:bounded_density} is the best
possible up to a constant $C$. Indeed, a simple counting argument shows that if $e(G) = o(n^{2 - 1/d})$ then the number of graphs with density at most $d$ which can possibly appear in $G$ is less than the total number of such graphs. Moreover, we obtain such a lower bound even when restricting $\cH_d(n)$ to graphs with bounded maximum degree. This is summarised in the following proposition.

\begin{proposition} \label{prop:lower_bound}
    For every $d \in \mathbb{Q}$ ($d \ge 1$) there exists $D \in \mathbb{N}$ and $\alpha > 0$ such that the following holds for every sufficiently large $n$. If $G$ is an $\cH_d^D(n)$-universal graph then
    $$
        e(G) \ge \alpha n^{2 - 1/d},
    $$
    where $\cH_d^D(n)$ is the family of all graphs $H \in \cH_d(n)$ with maximum degree at most $D$.
\end{proposition}

A careful reader will notice that we require $d > 1$ in Conjecture \ref{conj:bounded_density}. Indeed, if a graph $H$ has
density at most $1$ then each connected component of $H$ contains at most one cycle, thus $H$ is almost a forest. For the family of all forests it is known that $\Theta(n \log n)$
edges are both necessary and sufficient \cite{chung78tree,chung83spanning}. This seems to be an artifact of the fact that having only $\Theta(n)$ edges is simply too restrictive with how we can arrange them, which we believe is not the case when we have $O(n^{1 + \eps})$ edges for any constant $\eps > 0$. While this justification is vague, one can draw analogy with the theory of random graphs, where the multiplicative $\log(n)$ factor becomes unnecessary when moving from Hamilton cycles to, say, powers of Hamilton cycles \cite{kahn21hamilton} (signifying the difference between unicyclic graphs and those of density $d > 1$).

A reader familiar with random graph theory \cite{bollobas81threshold} will notice that a random graph with
$n$ vertices and $\omega(n^{2 - 1/d})$ edges is likely to contain a given graph
$H$ with $m(H) \le d$, provided $v(H)$ is significantly smaller than $n$.
However, it is known that if $v(H)$ is large then, in some cases such as when $H$ is a collection of many triangles, a significantly denser random
graph is needed in order for $H$ to appear. It is interesting that from
the point of view of constructing universal graphs, this phenomena does not
happen.

A recent result of Allen, B\"{o}ttcher, and Liebenau \cite{allen2023universality} establishes the bound in Conjecture \ref{conj:bounded_density} (up to a $\log^{2/d}(n)$ factor) in the special case where we restrict attention to the graphs in $\cH_d(n)$ which are also $d$-degenerate (for $d$ an integer). Moreover, using the fact that a graph of density $d$ is $\lfloor 2d \rfloor$-degenerate, their result also implies an upper bound of order
$$
    O_d\left( n^{2 - 1 / \lfloor 2d \rfloor} \log^{1/(2d) + o(1)}(n) \right).
$$
for $\cH_d(n)$-universality. Our first main result, Theorem \ref{thm:integer_unbounded}, significantly improves this starting from $d \ge 1.5$.

\begin{theorem} \label{thm:integer_unbounded}
    For every $n \in \mathbb{N}$ and $d \in \mathbb{Q}$, $d > 1$, there exists a graph $G$ with
    $$
        e(G) \le C n^{2 - 1/(\lceil d \rceil +1)}
    $$
    edges which is $\mathcal{H}_d(n)$-universal, where $C = C(d)$.
\end{theorem}

As a further support towards Conjecture \ref{conj:bounded_density}, we consider the family $\cH_d^D(n)$ of graph $H \in \cH_d(n)$ with maximum degree $D$. In this case, we get a nearly-optimal bound. 

\begin{theorem} \label{thm:rational}
    For every $D, n \in \mathbb{N}$ and $d \in \mathbb{Q}$, $d > 1$, there exists a graph $G$ with
    $$
        e(G) \le n^{2 - 1/d} \cdot 2^{C \sqrt{\log n}}
    $$
    edges which is $\mathcal{H}_q^D(n)$-universal, where $C = C(D, d)$.
\end{theorem}

Finally, in the case where $d$ is an integer, we obtain an optimal bound.

\begin{theorem} \label{thm:integer}
    For every $D, n \in \mathbb{N}$ and $d \in \mathbb{N}$, there exists a
    graph $G$ with 
    $$
        e(G) \le C n^{2 - 1/d}
    $$
    edges which is $\mathcal{H}_d^D(n)$-universal, where $C = C(D, d)$.
\end{theorem}

Let us briefly compare these results with the previous ones. Alon and Capalbo \cite{alon08optimal} constructed a graph with $O_{D}(n^{2 - 2/D})$ edges which is universal for the family of all $n$-vertex graphs with maximum degree at most $D$. Theorem \ref{thm:integer} implies this result, and further generalizes it, in the case $D$ is even. In the case $D$ is odd, Theorem \ref{thm:rational} provides a bound which is by a factor of $2^{O(\sqrt{\log n})}$ weaker than the one from \cite{alon08optimal}. However, in comparison with the proof from \cite{alon08optimal} which relies on the fact that bounded degree graphs can be decomposed into path-like pieces, we use a much more general decomposition (Lemma \ref{lemma:matroid}) which applies to all graphs with bounded density. Moreover, Theorem \ref{thm:integer} improves a result of the fifth author \cite{nenadov2016ramsey} on universality of $d$-degenerate graphs with bounded degree, to an optimal one. That being said, improving Theorem \ref{thm:rational} is a natural first step towards Conjecture \ref{conj:bounded_density}.

\begin{conjecture} \label{conj:bounded_deg_density}
    For every $D, n \in \mathbb{N}$ and $d \in \mathbb{Q}$, $d \ge 1$, there exists a graph $G$ with
    $$
        e(G) \le C n^{2 - 1/d}
    $$
    edges which is $\mathcal{H}_d^D(n)$-universal, where $C = C(D, d)$.
\end{conjecture}

Unlike in the case of graphs with arbitrarily large degree, an additional multiplicative $\log(n)$ factor is not needed even in the case of forests \cite{friedman87tree}. Thus we can relax the condition to $d \ge 1$.

Finally, let us briefly note that our constructions of universal graphs are based on a \emph{product construction} first used by Alon and Capalbo \cite{alon07sparsebounded,alon08optimal}, refining an earlier approach by Alon, Capalbo, Kohayakawa, R\"odl, Ruci\'nski, and Szemer\'edi \cite{alon2001near}. 
The proofs here also apply some of the ideas of Beck and Fiala~\cite{beckfiala} from Discrepancy Theory, results of  Feldman, Friedman and Pippenger~\cite{feldman88nonblocking} (see also
\cite{draganic22rolling}) from the theory of nonblocking networks, and random walks on expanders, together with the Matroid Decomposition Theorem of Edmonds
\cite{edmonds65matroid}.

The paper is organised as follows. In the next section we prove that the bound in Conjecture \ref{conj:bounded_density}, if true, is optimal even if we only restrict attention to $\cH_d^D(n) \subseteq \cH_d(n)$, where the bound $D$ on the maximum degree depends on $d$. Section \ref{sec:preliminaries} collects some results used in two or all three proofs. We then proceed with proofs of our main theorems. Comments on differences between proofs are given when appropriate. Throughout the paper we assume, whenever this is needed, that the parameter $n$ is sufficiently large as a function of any other parameter. To simplify the presentation we omit all floor and ceiling signs whenever they are not crucial.

\section{Lower bound} \label{sec:lower_bound}

Consider a (fixed) graph $F$ such that $m(F) = e(F) / v(F)$. Such graphs are
called \emph{balanced}. Let $n \in \mathbb{N}$ be sufficiently large and
divisible by $v(F)$. We obtain a large family of $n$-vertex graphs
$H$ with $m(H) = m(F)$ as follows. Set $V(H) = V_1 \cup \ldots \cup V_{v(F)}$,
where $V_i$'s are disjoint sets of size $n / v(F)$, and for each $ij \in
E(G)$ put a perfect matching between $V_i$ and $V_j$. The resulting graph
$H$ is called a \emph{lift} of $F$. It is not difficult to check that regardless of which perfect
matching we choose, we have $m(H) = m(F)$. The number of such (labelled) graphs $H$ 
is  
\begin{equation} \label{eq:lift_count}
    \left( \left( \frac{n}{v(F)} \right)! \right)^{e(F)} > \left(
    \frac{n}{3v(F)} \right)^{n e(F) / v(F)}.
\end{equation}
We use this bound to prove Proposition \ref{prop:lower_bound}.

\begin{proof}[Proof of Proposition \ref{prop:lower_bound}]
The result of R\'ucinski and Vince~\cite{rucinski86balanced} implies that for every $d \in \mathbb{Q}$, $d \ge 1$, there exists a balanced graph $F$ with $m(F) = d$. As a lift $H$ of $F$ has maximum degree $D = v(F)$, we have $H \in \cH_d^D(n)$.

Suppose that a graph $G$ contains every lift of $F$ of order $n$, and let $M =
e(G)$. As every lift contains exactly $n e(F) / v(F)$ edges, by \eqref{eq:lift_count} we necessarily
have
\begin{equation} \label{eq:M_lower_bound}
  \binom{M}{n e(F) / v(F)} n! > \left( \frac{n}{3v(F)} \right)^{n e(F) / v(F)},
\end{equation}
as otherwise there is a lift of $F$ which does not appear in $G$. Note that the
$n!$ term on the left hand side takes into account that every choice of $n e(F) /
v(F)$ edges accounts for at most $n!$ different labeled subgraphs. We can further
upper bound the left hand side of \eqref{eq:M_lower_bound} as follows:
$$
    \binom{M}{n e(F) / v(F)} n! < \left( \frac{3M}{n} \right)^{n e(F) / v(F)}
    n^n = \left( \frac{3 M}{n^{1 - v(F) / e(F)}} \right)^{n e(F) / v(F)}.
$$
Comparing this with the right hand side of \eqref{eq:M_lower_bound}, we conclude
$$
    M > \frac{1}{9 v(F)} n^{2 - v(F) / e(F)} = \frac{1}{9 v(F)} n^{2 - 1/m(F)}.
$$
\end{proof}

\section{Preliminaries}
\label{sec:preliminaries}

In the following lemma we identify a graph with its edge set. We say that a
graph is \emph{unicyclic} if it contains at most one cycle. 
The proofs of all our three main theorems are based on the decomposition 
given by the following lemma. Its proof applies some basic 
results from Matroid Theory, see, e.g, \cite{Welsh} for the relevant notions.

\begin{lemma} \label{lemma:matroid}
  Let $H$ be a simple graph 
  satisfying $m(H) \geq 1$ and $H^{(b)}$ a multigraph obtained from $H$ by
  duplicating each edge $b \in \mathbb{N}$ times. Then there exists a
  partition $H^{(b)} = H_1 \cup \ldots \cup H_k$, where $k = \lceil b \cdot m(H) \rceil$,
  such that, for every $i \in [k]$, each component of $H_i$ is a simple unicyclic graph.
\end{lemma}
\begin{proof}
  It suffices to prove the statement in the case $H$ is connected.
  
  Let $\mathcal{B}$ be the family of all spanning subgraphs $B \subseteq H^{(b)}$
  which are simple, connected, and unicyclic. We will shortly prove that $\mathcal{M} =
  (E(H^{(b)}), \mathcal{B})$ is a matroid, with $\mathcal{B}$ being the family of
  bases of $\mathcal{M}$. Assume for now that this is true. Consider some $H'
  \subseteq H^{(b)}$, and let $H''$ be the graph obtained from $H'$ by removing
  duplicate edges. Then the rank $r(H')$ of $H'$ in $\mathcal{M}$ is $v(H'')$
  if $H''$ contains a cycle, and $v(H'') - 1$ otherwise (and note that $v(H') =
  v(H'')$). In the former case we have
  $$
    |H'| / r(H') = e(H') / r(H') \le b \cdot e(H'') / v(H'') \le b \cdot m(H).
  $$
  In the latter case we necessarily have $e(H'') < v(H'')$, thus 
  $$
    e(H') / r(H') \le b \cdot (v(H'') - 1) / (v(H'') - 1) = b \leq b \cdot m(H).
  $$
  By a result of Edmonds \cite{edmonds65matroid}, one can cover $H$ with
  $\lceil b \cdot m (H) \rceil$ disjoint independent sets from $\mathcal{M}$,
  which proves the lemma.

  Let us verify that $\mathcal{M}$ is indeed a matroid. Consider some distinct
  bases $X, Y \in \mathcal{B}$, and let $e \in X \setminus Y$. We aim to find
  $e' \in Y \setminus X$ such that $(X \setminus \{e\}) \cup \{e'\} \in \mathcal{B}$.  If
  $e$ lies on the cycle of $X$, then $X \setminus \{e\}$ is a tree thus adding
  any edge $e' \in Y \setminus X$ produces a spanning unicyclic graph.
  Otherwise, $X \setminus \{e\}$ breaks $X$ into two components, a tree $X_1$
  and a unicyclic $X_2$. As $e \notin Y$, there has to exist $e' \in Y \setminus
  X$ with one endpoint in $V(X_1)$ and the other in $V(X_2)$.  Therefore, $(X
  \setminus \{e\}) \cup \{e'\}$ is again a spanning unicyclic subgraph.
\end{proof}

In the proofs of Theorem \ref{thm:integer_unbounded} and Theorem \ref{thm:integer}, we make use of the following simple known lemma (see, e.g., \cite[Lemma 2.2]{alon02treecut}).

\begin{lemma}[\cite{alon02treecut}] \label{lemma:tree_split}
  For every forest $F$ with $n$ vertices and every $r \in \mathbb{N}$, there
  exists a subset $R \subseteq V(F)$ of size $|R| \le r$ such that each
  connected component in $F \setminus R$ is of size at most $n/r$.
\end{lemma}

In the proofs of Theorems \ref{thm:integer_unbounded} and \ref{thm:rational}, we use the notion of a graph \emph{blowup}, thus we define it here for the reference.

\begin{definition} \label{def:blowup}
    Given a graph $G$ and $b \in \mathbb{N}$, we define a \emph{$b$-blowup} of $G$ to be a graph $\Gamma$ on the vertex set $V(\Gamma) = \bigcup_{u \in V(G)} V_u$, where each $V_u$ is of size $b$ and all the sets are disjoint, and there is an edge between $a \in V_u$ and $b \in V_w$ iff $uv \in E(G)$ or $u = v$. In particular, $\Gamma$ has  $v(G) b$ vertices and $v(G) \binom{b}{2} + e(G) b^2$ edges.
\end{definition}

In the proofs of Theorem \ref{thm:rational} and Theorem \ref{thm:integer} we make use of the so-called \emph{$(n, t, \lambda)$-graphs}. These are $t$-regular graphs with $n$ vertices and the second largest absolute eigenvalue at most $\lambda$. For small $\lambda$ such graphs are known to be good expanders. We use the following result of the first author \cite{alon21expander} which provides an explicit construction of such graphs for an almost optimal value of $\lambda$ (note that a construction of Lubotzky, Phillips, and Sarnak \cite{lubotzky88ramanujan} can be used as well, even though it does not provide a construction for every $t$).

\begin{theorem} \label{thm:ndl}
    For every $t \in \mathbb{N}$ and $n \ge n_0(t)$ such that $nt$ is even, there exist an explicit construction of an $(n, t, \lambda)$-graph with $\lambda \le 3 \sqrt{t}$.
\end{theorem}

It is worth noting that the proof in \cite{alon21expander} and the known results about the Linnik problem imply that $n_0(t) \leq t^{O(1)}$. In particular, this is relevant for the proof of Theorem \ref{thm:rational}.

\section{Graphs with bounded density}

We need the following lemma from Discrepancy Theory. The proof 
applies the Beck-Fiala method~\cite{beckfiala}. We only use it with $q = 1 / 2$, however, as this does not make the proof any easier, we state it in greater generality.

\begin{lemma} \label{l21}
The following holds for any positive integers $t$, $d$ and real $q \in [0,1]$ 
Let $\vv_1, \vv_2, \ldots ,\vv_t$ be a sequence of vectors in $\mathbb{R}^d$, each having $\ell_1$-norm at most $1$. Then there is a subset $I \subseteq \{1,2,\ldots ,t\}$ so that
$$
    \left\lVert\sum_{i \in I} \vv_i - q  \sum_{i=1}^t \vv_i \right\rVert_{\infty} < 1.
$$
\end{lemma}
\begin{proof}
Associate each vector  $\vv_i=(v_{i1},v_{i2}, \ldots ,v_{id})$ with a real
variable $x_i \in [0,1]$. Starting with $x_i=q$ for all $i$, we describe an algorithm for rounding each $x_i$ to either $0$ or $1$ without changing the sum $\sum x_i v_i$ by much. During this algorithm, call a variable
$x_i$ \textit{floating} if $x_i$ lies in the open interval $(0,1)$,  otherwise
(that is, if $x_i \in \{0,1\}$) call it \textit{fixed}. Once a variable becomes
fixed it will stay  fixed until the end.  In each phase, the algorithm  
proceeds as follows. If all variables are fixed, terminate. Otherwise, let
$F\subseteq \{1,2,\ldots, t\}$ denote the set of all indices of floating
variables and consider the following linear system of equations in these
variables.

For every coordinate $j \in \{1,2,\ldots, d\}$ for which  
\begin{equation} \label{e21}
    \sum_{i \in F} |v_{ij}| >1, 
\end{equation}
include the equation
\begin{equation} \label{e22}
    \sum_{i=1}^t x_i v_{ij} = q \sum_{i=1}^t v_{ij}.
\end{equation} 
Note that only the floating variables $\{x_i \colon i \in F\}$ are considered as variables at this point; the fixed variables are already fixed and are treated as constants. During the algorithm, we maintain the property that for every coordinate $j$ for which (\ref{e21}) holds, the equality  (\ref{e22}) holds as well. This is certainly true at the beginning, when $x_i=q$ for all $i$.

By the assumption about the $\ell_1$-norm of the vectors $v_i$, we have
$$
  \sum_{i \in F} \sum_{j=1}^d |v_{ij}| \leq |F|
$$
and therefore the number of
indices $j$ for which (\ref{e21}) holds is strictly smaller than $|F|$. Thus
there are more variables than equations and hence there is a line of solutions.
One can move along this line starting with the existing point on it (that
corresponds to the current value of the floating variables $x_i$) until the
first point in which at least one of the floating variables $x_i$ becomes $0$
or $1$. Fix this variable (as well as any other floating variables that become
$0$ or $1$, if any), and continue with the next phase. Note that the desired
property is maintained since no new coordinate $j$ can satisfy~(\ref{e21}) as
$F$ gets smaller.

Since each phase fixes at least one floating variable, the algorithm must
terminate when all the variables $x_i$ are fixed. Now, for a coordinate $j$,
consider the first phase when condition \eqref{e21} is not satisfied. The
equality \eqref{e22} holds at this phase since it is the first time \eqref{e21}
is not satisfied. Moreover, for the rest of the algorithm, the value of the sum
$\sum_{i=1}^t x_i v_{ij}$ can only change by strictly less than $\sum_{i \in F}
|v_{ij}| \leq 1$, since only the floating variables change after that. This
shows that upon termination, for every index $j\in \{1,2,\ldots, d\}$,
$$
    \left|\sum_{i=1}^t x_i v_{ij} - q \sum_{i=1}^t v_{ij}\right|<1.
$$
Therefore the set $I =\{i: x_i =1\}$  of all indices in which the final value
of $x_i$ is $1$ satisfies the conclusion of the lemma.  
\end{proof}

By repeated application of the previous lemma, we get the following.

\begin{corollary}
\label{c22}
The following holds for any three positive integers $t,d,m=2^k$. Let $\vv_1,\vv_2,
\ldots, \vv_t$ be a sequence of vectors in $\mathbb{R}^d$, each having $\ell_1$-norm at
most $1$. Then there is a partition of the vectors into $m$ pairwise disjoint
sets $\{v_i \colon i \in I_p\}$, $1 \leq p \leq m$, where $[t]=I_1 \cup  \ldots  \cup
I_m$, the sets $I_j$ are pairwise disjoint, and for every $p \in [m]$ we have
$$
\left\lVert\sum_{i \in I_p} \vv_i - \frac{1}{m} \sum_{i=1}^t \vv_i \right\rVert_{\infty} \le \sum_{i = 0}^{k-1} 2^{-i} < 2.
$$
\end{corollary}
\begin{proof}
We prove the statement by induction on $k$. For $k = 1$ it is equivalent to Lemma \ref{l21} with $q = 1/2$. Suppose that it hold for $m = 2^{k-1}$, for $k \ge 2$. We show that then it also holds for $m = 2^k$. 

Apply Lemma \ref{l21} with $q = 1 / 2$ to split the vectors into two collections, $[t] = C_1 \cup C_2$, such that for $i \in \{1,2\}$ we have
\begin{equation} \label{eq:linf_Ci}
\left\lVert\sum_{i \in C_i} \vv_i - \frac{1}{2} \sum_{i=1}^t \vv_i \right\rVert_{\infty} < 1.
\end{equation}
By the induction hypothesis, there is a partition $C_1 = I_1 \cup \ldots \cup I_{m/2}$ such that for each $p \in [m/2]$ we have
\begin{equation}     \label{eq:linf_Ci_ind}
\left\lVert\sum_{i \in I_p} \vv_i - \frac{2}{m} \sum_{i \in C_1} \vv_i \right\rVert_{\infty} \le \sum_{i = 0}^{k-2} 2^{-i}.
\end{equation}
By the triangle inequality, from \eqref{eq:linf_Ci} and \eqref{eq:linf_Ci_ind} we conclude
$$  
\left\lVert\sum_{i \in I_p} \vv_i - \frac{1}{m} \sum_{i = 1}^t \vv_i \right\rVert_{\infty} 
= 
\left\lVert 
    \left( \sum_{i \in I_p} \vv_i - \frac{2}{m} \sum_{i \in C_1} \vv_i \right) +
    \frac{2}{m} \left( \sum_{i \in C_1} \vv_i - \frac{1}{2} \sum_{i = 1}^t \vv_i \right)
\right\rVert_{\infty}
\le \sum_{i = 0}^{k-2} 2^{-i} + 2/m.
$$
The same argument applies to $C_2$, which gives a desired partition $[t] = I_1 \cup \ldots \cup I_m$.
\end{proof}

\begin{proof}[Proof of Theorem \ref{thm:integer_unbounded}] 
  Note that it suffices to prove the theorem for $d \in \mathbb{N}$.

  Let $m$ be the smallest power of $2$ that is at least  $n^{1/(d+1)}$, 
and suppose $m \geq 4$. First, form a graph $\Gamma$ on the
vertex set $[m]^d$ where two vertices $\vu = (u_1,u_2,\ldots ,u_d)$ and $\vv = (v_1,v_2,
\ldots ,v_d)$ are connected if $u_i=v_i$ for some $i \in [d]$. The graph $\Gamma^+$ is
the $(3m+3)$-blowup of $\Gamma$ (see Definition \ref{def:blowup}) together with another set $V^+$ of $2 dn/m$  vertices
and all edges incident with at least one of them. Note that $\Gamma^+$ has
$O(dn^{1-1/(d+1)})$ edges. We proceed to show that it is $\cH_d(n)$-universal.

Consider some $H \in \cH_d(n)$. Let $H = H_1 \cup \ldots \cup H_d$
be a decomposition given by Lemma \ref{lemma:matroid} (with $b = 1$), and
recall that each component, of each $H_i$, is unicyclic. First form $R' \subset V(H)$ as follows: For every $i \in [d]$ and every component of $H_i$ of size at least $m$, take one vertex from a cycle in that component (if such exist). This adds up to at most $d n / m$ vertices. Next, by applying Lemma \ref{lemma:tree_split} with $F = H_i \setminus R'$ (which is now a forest) and  $r = n/m$ for each $i \in [d]$, we obtain a set $R \subseteq V(H)$ of size $|R|
\le dn/m$ such that each connected
component of $H_i \setminus (R \cup R')$ is of size at most $m$.
All the vertices of $R \cup R'$ will be mapped into $V^+$, thus we can set $H=H
\setminus (R \cup R')$ and $H_i = H_i \setminus (R \cup R')$. 

Let $\cC_i$ denote the family of connected components in $H_i$ (we identify a connected component by its vertex set). For each $h \in
V(H)$, let $c_i(h)$ denote the component $K \in \cC_i$ that contains $h$.  We
show, by induction on $i$, that there exist functions $\phi_i \colon \cC_i
\mapsto [m]$ such that for each $i \in [d]$ and every $\vv=(v_1,v_2, \ldots ,v_i)
\in [m]^i$ we have
\begin{equation}
\label{e23}
|S_\vv| \leq \frac{n}{m^i} +2m +3
\end{equation}
where 
$$
S_\vv=\{h \in V(H): \phi_1(c_1(h))=v_1, \phi_2(c_2(h))=v_2, \ldots ,
\phi_i(c_i(h))=v_i \}.
$$
Once we have this for $i=d$, by injectively mapping each $S_\vv$ into 
the blowup of $\vv$, for $\vv \in [m]^d$, we obtain an embedding
of $H$ in $\Gamma^+$. 

Inequality \eqref{e23} trivially holds for $i=0$.  Suppose that (\ref{e23})
holds for some $i-1$, for $i \in [d]$.  We show we can find $\phi_i$ so that it
holds for $i$. For each connected component $K \in \cC_i$
define a vector $\vv_K$ of length $m^{i-1}$ indexed by the vectors $\vu
=(u_1,u_2, \ldots ,u_{i-1}) \in [m]^{i-1}$ as follows:  The coordinate of $\vv_K$
indexed by $\vu$ is the number of  vertices $h \in K$ such that
$$
\phi_1(c_1(h))=u_1,\; \phi_2(c_2(h))=u_2, \ldots 
,\; \phi_{i-1}(c_{i-1}(h))=u_{i-1}.
$$
Note that the $\ell_1$-norm of each $\vv_K$ is the number of vertices of
$K$, which is at most $m$. In addition, the sum of all the vectors
$\vv_K$ in each coordinate $\vu$ is exactly $|S_\vu|$, which, by the
induction hypothesis, is at most $n/m^{i-1}+2m+3$. 

By Corollary \ref{c22}
these vectors can be partitioned into $m$ pairwise
disjoint collections so that the sum of the vectors in each
collection, and with respect to each coordinate, is at most
$$
\frac{n/m^{i-1}+2m+3}{m}+2m \leq n/m^i+2m+3.
$$
The value of $\phi_i(K)$ is now set to be the index of the
collection containing $K$, implying the required inequality for
$i$ and completing the proof.
\end{proof}

\section{Graphs with bounded density and degree}

The basic idea behind the proofs of Theorem \ref{thm:rational} and 
Theorem \ref{thm:integer} is similar to that 
in the proof of Theorem \ref{thm:integer_unbounded}. In Theorem \ref{thm:integer_unbounded} we obtain $n^{-1/(d+1)}$ instead of desired $n^{-1/d}$ because, in each of the $d$ steps, we assign the same coordinate to all the vertices of a connected component in $H_i$. Intuitively, if same vertices of $H$ belong to the same connected component across each $H_i$, this is not sufficient to disambiguate them and we are forced to take a small blowup at the end.

When $H$ has bounded maximum degree, we avoid this by using the following idea, at least in the case $d \in \mathbb{N}$: Our basic graph $\Gamma$ is again defined on the vertex set $[m]^d$ (now with $m \approx n^{1/d}$), however this time we connect $\vv, \vu \in [m]^d$ by an edge if some $v_i$ and $u_i$ are connected by an edge in a bounded-degree expander $G$ on the vertex set $[m]$, which we fix upfront. Instead of mapping all the vertices of one component of $H_i$ into a single coordinate, we disperse them across $[m]$ by using edges of the expander $G$. In Theorem \ref{thm:integer} we can make this approach disambiguate all the vertices of $H$, thus avoiding the use of a final blowup all together. In Theorem \ref{thm:rational} the number of vertices which are pairwise ambiguous ends up being of order $2^{O(\sqrt{\log n})}$, thus we take a very small blowup at the end -- significantly smaller than in the proof of Theorem \ref{thm:integer_unbounded} -- to deal with this.

The proof of Theorem \ref{thm:rational}, presented next, borrows ideas of using random walks in expanders from \cite{alon07sparsebounded}. One significant difficulty in the proof of Theorem \ref{thm:rational} is that we are not able to split $H_i$ into small connected components and we have to deal with the whole $H_i$ at once, which further emphasizes dispersion via expanders. The proof of Theorem \ref{thm:integer} generalizes the approach from \cite{alon08optimal} 
from embedding paths in expanders, in a specific way, 
to embedding bounded-degree trees. This is done using some of the
ideas in \cite{feldman88nonblocking} and \cite{draganic22rolling}.

\subsection{Density bounded by a rational}

We use the following well known property of random walks on expanders,
see, e.g., \cite{HLW}.

\begin{lemma} \label{lemma:random_walk}
    Let $G$ be an $(n, t, \lambda)$-graph, and consider a random walk starting in a given vertex $v \in V(G)$. The probability that after exactly $\ell$ steps we finish in a vertex $w \in V(G)$ is at most
    $$
        1/n + (\lambda / t)^\ell. 
    $$
\end{lemma}

\paragraph{Randomized tree homomorphism.} Given a tree $T$ with the designated root $r$ and a graph $G$, we use the following randomized procedure for constructing a homomorphism $\phi \colon T \hookrightarrow G$:
\begin{enumerate}[(i)]
    \item Consider any ordering $h_1, \ldots, h_n$ of $V(T)$ such that $h_1 = r$ and, for each $i \ge 2$, $h_i$ has exactly one neighbour within $\{h_1, \ldots, h_{i-1}\}$.
    \item Take $s_1 \in V(G)$ to be some upfront chosen vertex in $V(G)$.
    \item For $i = \{2, \ldots, n\}$, sequentially, take $s_i$ to be a neighbour of $s_j$ in $G$ chosen uniformly at random, where $j < i$ is a unique index such that $h_j h_i \in T$.
\end{enumerate}
The homomorphism is then given by $\phi(h_i) := s_i$. Note that the ordering of the vertices $h_2, \ldots, h_n$ plays no role in the distribution of $\phi$, as long as  each vertex other than $h_1$ has exactly one predecessor.

\begin{lemma} \label{lemma:random_tree}
    Let $G$ be an $(m, t, 3\sqrt{t})$-graph where $t = 2^{\sqrt{\log n}}$ and $n \ge m$. Suppose $T$ is a tree with the root $r$, and $U \subseteq V(T) \setminus \{r\}$ a subset such that every two $t, t' \in U \cup \{r\}$ are at distance at least $16\sqrt{\log n}$ in $T$. Let $\phi$ be a random homomorphism $\phi \colon T \hookrightarrow G$ obtained by the described procedure. Then, for any $v \in V(G)$, the size of the set
    $$
        U_v = \{u \in U \colon \phi(u) = v \} 
    $$
    is stochastically dominated by a binomial random variable $B(|U|,1/m + 1/n^3)$.
\end{lemma}
\begin{proof}
    Let $u_1, \ldots, u_k$ be an ordering of the vertices in $U$ such that if $u_j$ is closer to $r$ than $u_i$, then $j < i$. Let $P_1$ be the path from $r$ to $u_1$ and set $x_1 = r$. For each $2 \le i \le k$, define the path $P_i$ as follows:
    \begin{itemize}        
        \item Let $x_i \in V(T)$ be the first vertex on a path from $u_i$ to $r$ which belongs to $\bigcup_{j < i} V(P_j)$;
        \item Set $P_i$ to be the path from $x_i$ to $u_i$.
    \end{itemize}
    Importantly, for every $i \in [k]$ we have $|P_i| \ge 8 \sqrt{\log n}$. Let us quickly prove this. As $x_i \in \bigcup_{j < i} V(P_j)$ we have $x_i \in V(P_j)$, for some $j < i$. That implies $u_j$ is not further from $r$ than $u_i$, thus the path from $x_i$ to $u_j$ is not larger than $|P_i|$. Therefore $u_i$ and $u_j$ are at distance at most $2|P_i|$, which gives the desired lower bound.
    
    We now describe an equivalent way of generating $\phi$:
    \begin{enumerate}[(i)]
        \item Set $\phi(r)$ to be the upfront chosen vertex $s_1$ in $V(G)$.
        \item For each $i \in [k]$, sequentially, extend the partial mapping $\phi$ to $V(P_i) \setminus \{x_i\}$ by taking a random walk of length $|P_i|$ which starts in $\phi(x_i)$.
        \item Let $f_1, \ldots, f_{k'}$ be an ordering of the vertices in
        $$
            V_P = V(T) \setminus \bigcup_{i \in [k]} V(P_i)
        $$
        such that each $f_i$ has exactly one neighbour $f_i' \in V_P \cup \{f_1, \ldots, f_{i-1}\}$. Sequentially, for $i \in [k']$, extend $\phi$ to $f_i$ by taking a random neighbour of $\phi(f_i')$.
    \end{enumerate}
    By Lemma \ref{lemma:random_walk} we have
    $$
        \Pr[\phi(u_i) = v \mid \phi(u_1), \ldots, \phi(u_{i-1})] \le 1/m + 
 (4 / \sqrt{t})^{|P_i|} < 1 / m + 1/n^3,
    $$
    thus the conclusion of the lemma follows.
\end{proof}

Finally, we make use of the following lemma which allows us to treat unicyclic graphs as trees.

\begin{lemma} \label{lemma:tree_square}
    Let $H$ be a connected unicyclic graph. Then there exists a tree $T$ on
    the same vertex set such that $\Delta(T) \le \Delta(H)$ and $H \subseteq
    T^2$.
\end{lemma}
\begin{proof}
    Let $v_1, \ldots, v_k \in V(H)$ be the vertices along the cycle in $H$.
    Form a tree $T$ by removing the edges on the cycle in $H$, and adding the
    edges on the path $v_1 v_k v_2 v_{k-1} v_3 v_{k-2} \ldots v_{k'}$, where
    $k' = \lceil (k+1)/2 \rceil$.
\end{proof}

We are ready to prove Theorem \ref{thm:rational}.

\begin{proof}[Proof of Theorem \ref{thm:rational}]
    Suppose $d = a / b$, for some $a, b \in \mathbb{N}$ with $a \ge b$. Let $m = n^{1/a}$ and $t = 2^{\sqrt{\log n}}$, and let $G$ be an 
	$(m, t, 3\sqrt{t})$-graph on the vertex set $[m]^a$ (see Theorem \ref{thm:ndl}). We form the graph $\Gamma$ as follows: $V(\Gamma) = [m]^a$, and two vertices
    $\vv = (v_1, \ldots, v_a)$ and $\vw = (w_1, \ldots, w_a)$ are connected by an
    edge iff there exist at least $b$ distinct indices $i_1, \ldots, i_b \in
    [a]$ such that $v_j w_j \in G^2$, for each $j \in \{i_1, \ldots, i_b\}$.
    Finally, take $\Gamma^+$ to be a $(2^{C \sqrt{\log n}})$-blowup of
    $\Gamma$, for $C$ being a sufficiently large constant. The graph $\Gamma^+$ has
    $$
        O\left( n^{2 - b/a} \cdot 2^{2C \sqrt{\log n}} \right)
    $$
    edges. It remains to show that $\Gamma^+$ is $\mathcal{H}_{d}^D(n)$-universal.

    Consider some $H \in \mathcal{H}_{d}^D(n)$. Applying
    Lemma \ref{lemma:matroid} with $b$, we obtain subgraphs $H_1, \ldots, H_a \subseteq H$ such that each connected component, of each $H_i$, is unicyclic, and each edge $e \in H$ belongs to exactly $b$ of these subgraphs. By Lemma \ref{lemma:tree_square}, for each component $H_i$ there exists a forest $T_i$ such that $H_i \subseteq T_i^2$ and $\Delta(T_i) \le D$. Therefore, any homomorphism of $T_i$ into $G$ is also a homomorphism of $H_i$ into $G^2$. By adding edges across leaves of some components in $T_i$, we can assume that $T_i$ is a spanning tree on the vertex set $V(H)$. Let $r \in V(H)$ be an arbitrary vertex which will serve as the
    root of every tree $T_i$. 
    
    Form an auxiliary graph $A$ by taking an edge between $h, h' \in V(H)$ iff
    they are at distance at most $16 \sqrt{\log n}$ in some $T_i$. Then
    $$
        \Delta(A) \le a D^{16 \sqrt{\log n}}.
    $$
    Take $U_0 \subseteq V(H)$ to be the set of all vertices in $V(H)$ which are neighbours of $r$ in $A$, together with $r$ itself. Using Hajnal-Szemer\'edi theorem, partition $V(H) \setminus U_0$ into independent sets $U_1, \ldots, U_{\Delta(A) + 1}$ in $A$.

    Our goal is to find homomorphisms $\phi_i \colon T_i \hookrightarrow G$ such that, for each $i \in [a]$, $\mathbf{v} = (v_1, \ldots, v_i) \in [m]^i$ and $j \in \{1, \ldots, \Delta(A) + 1\}$, we have
    \begin{equation} \label{eq:s_size}
        |S_{\mathbf{v}}^j| \le \max\{2n^{(a - i)/a}, 4 \log n\}, 
    \end{equation}
    where
    $$
        S_{\mathbf{v}}^j = \{ h \in U_j \colon \phi_1(h) = v_1, \ldots, \phi_i(h) = v_i \}.
    $$
    This implies that, for every $\mathbf{v} = (v_1, \ldots, v_a) \in [m]^a$, the set
    $$
        S_{\mathbf{v}} = \{ h \in V(H) \colon \phi_1(h) = v_1, \ldots, \phi_a(h) = v_a \}
    $$
    is of size
    $$
        |S_{\mathbf{v}}| \le |U_0| + (\Delta(A) + 1) \cdot 4 \log n < 2^{C \sqrt{\log n}},
    $$
    thus by injectivelly mapping $S_{\mathbf{v}}$ into the blowup of $\vv$, we obtain a copy of $H$ in $\Gamma^+$.

    Suppose that we have found $\phi_1, \ldots, \phi_{i-1}$ such that \eqref{eq:s_size} holds. Let $\phi_i \colon T_i \hookrightarrow G$ be a random homomorphism generated as described at the beginning of this section. By Lemma \ref{lemma:random_tree} and standard estimates of the binomial distribution, this holds for one particular choice of $\mathbf{v} = (v_1, \ldots, v_i)$ and $j \in \{1, \ldots, \Delta(A) + 1\}$ with probability at least $1 - 1 / n^2$. 
Therefore, by the union-bound, it holds for all choices with positive probability, thus a desired homomorphism exists.
\end{proof}

\subsection{Density bounded by an integer}

The following lemma replaces the use of randomness in the proof of Theorem \ref{thm:rational} and is the core of the proof of Theorem \ref{thm:integer}.

\begin{lemma} \label{lemma:tree}
    For every $D \in \mathbb{N}$ there exist $t \in \mathbb{N}$ and $\eps > 0$ such that the following holds. Let $G$ be an $(n, t, 3\sqrt{t})$-graph, and let $T$ be a tree with  $v(T) \le n / (3t)$ vertices and maximum degree $\Delta(T) \le D$. Then for any family of subsets $\{S_v \subseteq V(G)\}_{v \in V(T)}$ with $|S_v| \ge (1 - \eps)n$ for each $v \in V(T)$, there exists an embedding $\phi \colon T \rightarrow G$ such that $\phi(v) \in S_v$ for every $v \in V(T)$.
\end{lemma}

The main machinery underlying the proof of Lemma \ref{lemma:tree} is a result from the theory of \emph{nonblocking networks}, due to Feldman, Friedman, and Pippenger \cite[Proposition 1]{feldman88nonblocking}. An efficient algorithmic version of this result was obtained by Aggarwal et al.~\cite{aggarwal96optical}.

\begin{definition} \label{def:nonblocking}
Given $t, s \in \mathbb{N}$, we say that a bipartite graph $B = (V_1 \cup V_2, E)$ is \emph{$(t, s)$-nonblocking} if there exists a family $\cS$ of  subsets of $E$, called the \emph{safe} states, such that the following holds:
\begin{enumerate}[(P1)]
    \item $\emptyset \in \cS$,
    \item if $E'' \subseteq E'$ and $E' \in \cS$ then $E'' \in \cS$, and
    \item \label{prop:c} given $E' \in \cS$ of size $|E'| < s$ and a vertex $v \in V_1$ with $\deg_{E'}(v) < t$ (that is, $v$ is incident to less than $t$ edges in $E'$), there exists an edge $e = (v, w) \in E \setminus E'$ such that $E' \cup \{e\} \in \cS$ and $w$ is not incident to any edge in $E'$.
\end{enumerate}
\end{definition}

\begin{lemma}[\cite{feldman88nonblocking}] \label{lemma:nonblocking}
Let $B = (V_1 \cup V_2, E)$ be a bipartite graph and $a, t \in \mathbb{N}$. If
$$
    |N_B(X)| \ge 2t|X|
$$
for every $X \subseteq V_1$ of size $1 \le |X| \le 2a$, then $B$ is $(t, ta)$-nonblocking.
\end{lemma}



We are now ready to prove Lemma \ref{lemma:tree}.

\begin{proof}[Proof of Lemma \ref{lemma:tree}]
    Let $h_1, \ldots, h_r$ be an ordering of the vertices of $T$ such that for each $2 \le i \le r = v(T)$, $h_i$ has exactly one neighbour within $\{h_1, \ldots, h_{i-1}\}$. For $i \in \{r, \ldots, 1\}$, iteratively, define the set $A_i \subseteq S_{h_i}$ as follows: If $h_i$ does not have a neighbour within $\{h_{i+1}, \ldots, h_r\}$, set $A_i = S_{h_i}$; otherwise, let $R_i = \{ j > i \colon h_i h_j \in T \}$ and set
    $$
        A_i = \{ v \in S_{h_i} \colon |N_G(v, A_j)| \ge (1 - \beta)t \text{ for every } j \in R_i\},
    $$
    where $\beta > 0$ is a sufficiently small constant we will specify shortly. We show that each $A_i$ is of size $|A_i| \ge (1 - 2\eps)n$. This clearly holds for $i = r$. Suppose that it holds for $A_{i+1}, \ldots, A_r$, for some $1 \le i \le r - 1$. We show that it then holds for $A_i$ as well. We can assume $R_i \neq \emptyset$, as otherwise we are immediately done. Suppose, towards a contradiction, that $|A_i| < (1 - 2 \eps)n$. As $|R_i| \le D$, there exists $j \in R_i$ such that the set
    $$
        X = \{ v \in S_{h_i} \colon |N_G(v, A_j)| < (1 - \beta)t\}
    $$
    is of size $|X| \ge \eps n / D$. By the definition of $X$ and \cite[Theorem 9.2.4]{alon16book}, we have
    $$
        |X| (\beta - 2\eps)^2 t^2  \le \sum_{v \in X} \left(|N_G(v, A_j)| - (1 - 2\eps)t \right)^2 \le 9 t \cdot 2\eps n.
    $$
    For $t > 18 D / (\beta - 2\eps)^2$ this gives $|X| < \eps n / D$, thus a contradiction.

    Before we move to the embedding of $T$, we need another bit of preparation. Let $B$ be the bipartite graph on the vertex set $V(G) \times \{1,2\}$ where $(v,i)$ and $(w,j)$ are connected by an edge iff $vw \in G$ and $i \neq j$. By, e.g.\ \cite[Lemma 2.4]{alon20distributed}, for sufficiently small $\beta > 0$ we have
    $$
        |N_B(X \times \{1\})| \ge 4 \beta t |X| 
    $$
    for each $X \subseteq V(G)$ of size $|X| \le 2n/t$. Therefore, by Lemma \ref{lemma:nonblocking}, $B$ is $(2\beta t, n/t)$-nonblocking.
    

    We find distinct vertices $s_1, \ldots, s_r \in V(G)$ such that mapping $h_i$ to $h_i$ gives a copy of $T$ in $G$.  Throughout the procedure we maintain a \emph{safe} subset $E \subseteq B$ (see Definition \ref{def:nonblocking}), which is initially empty. First choose arbitrary $s_1 \in A_1$. Then, for each $2 \le i \le r$, sequentially, do the following:
    \begin{enumerate}[(i)]
        \item Let $j < i$ be the unique index such that $h_j h_i \in T$.
        \item \label{step:2} Obtain $E \subset E' \in \mathcal{S}$ by repeatedly applying \ref{prop:c} with $v = (s_j,1)$, such that at the end we have $\deg_{E'}((s_j,1)) = 2 \beta t$. 
        \item \label{step:3} Choose an edge $((s_j,1), (w,2)) \in E' \setminus E$ such that $w \in A_i \setminus \{s_1\}$. Set $s_i := w$ and $E = E \cup \{(s_i, 2)\}$.
    \end{enumerate}
    Note that $E$ remains a safe subset throughout the procedure. It is also evident from the description of the procedure that $\deg_E(v) \le \Delta(T)$ for every $v \in V(G) \times \{1\}$, and $|E| < v(T)$. As $v(T) + 2\beta h < n/(2h)$, step \ref{step:2} is well-defined. It remains to show that a desired edge in \ref{step:3} always exists. Consider some step $i$. As $s_j \in A_j$ and $\deg_{E'}(s_j) = 2 \beta t$, by the definition of $A_j$ we have
    $$
        |N_{E'}(s_j, A_i)| \ge \beta t,
    $$
    thus $\deg_E(s_j) \le D < \beta t - 1$ (by choosing $t$ to be sufficiently large) implies  the desired edge indeed exists. Note that we need to explicitly exclude $s_1$ in \ref{step:3} as $(s_1,2)$ is not an endpoint of any edge in $E$.
\end{proof}

Note that in the previous proof, one could as well use \cite[Theorem 2.8]{draganic22rolling}. For our purposes, we find Lemma \ref{lemma:nonblocking} to provide cleaner framework.


\begin{proof}[Proof of Theorem \ref{thm:integer}]
Let $t \in \mathbb{N}$ and $\eps > 0$ be as given by Lemma \ref{lemma:tree}. Furthermore, let $m = C n^{1/d}$, for sufficiently large $C$, and let $G$ be an $(m, t, 3\sqrt{t})$-graph 
on the vertex set $[m]$ (see Theorem \ref{thm:ndl}). We first construct $\Gamma$ as follows: $V(\Gamma_n) = [m]^d$ and two vertices $\vv = (v_1, \ldots, v_k)$ and $\vw = (w_1, \ldots, w_k)$ are connected by an edge iff there exists $i \in [d]$ such that $v_i w_i \in G$. Note that $\Gamma$ has $O(n)$ vertices and $O(n^{2 - 1/d})$ edges. Finally, we construct $\Gamma^+$ by adding a new set $V^+$ of $2d n / m$ vertices and adding all the edges incident to at least one vertex in $V^+$. The number of edges of $\Gamma^+$ remains $O(n^{2 - 1/d})$. We show that $\Gamma^+$ is $\mathcal{H}_{d}^D(n)$-universal.

Consider some $H \in \mathcal{H}_{d}^D(n)$, and let $H = H_1 \cup \ldots \cup H_d$ be a decomposition of $H$ given by Lemma \ref{lemma:matroid} (with $b = 1$). We clean-up $H_i$'s as in the proof of Theorem \ref{thm:integer_unbounded}. First form $R' \subset V(H)$ as follows: For every $i \in [d]$ and every component of $H_i$ of size at least $m$, take one vertex from a cycle in that component (if such exist). This adds up to at most $d n / m$ vertices. Next, by applying Lemma \ref{lemma:tree_split} with $F = H_i \setminus R'$ (which is now a forest) and  $r = n/m$ for each $i \in [d]$, we obtain a set $R \subseteq V(H)$ of size $|R| \le dn/m$ such that each connected
component of $H_i \setminus (R \cup R')$ is of size at most $m$.
All the vertices of $R \cup R'$ will be mapped into $V^+$, thus we can set $H=H
\setminus (R \cup R')$ and $H_i = H_i \setminus (R \cup R')$. 


We iteratively find homomorphisms $\phi_i \colon H_i \hookrightarrow G$ such that, for each $\mathbf{v} = (v_1, \ldots, v_i) \in [m]^i$, we have
\begin{equation} \label{eq:S_v}
    |S_{\mathbf{v}}| \le n^{(d - i)/d},
\end{equation}
where
$$
S_{\mathbf{v}} = \left\{ h \in V(H) \colon \phi_1(h) = v_1, \ldots, \phi_i(h) = v_i \right\}.
$$
Once we have this, a homomorphism $\phi \colon H \hookrightarrow \Gamma$ given by $\phi(h) = (\phi_1(h), \ldots, \phi_k(h))$ is an injection, thus $H$ is indeed a subgraph of $\Gamma$.

Suppose we have found homomorphisms $\phi_1, \ldots, \phi_{i-1}$, for some $i \in [d]$, such that \eqref{eq:S_v} holds. To find a desired homomorphism $\phi_i \colon H_i \hookrightarrow G$, we do the following. Take trees in $H_i$ one at a time, in an arbitrary order, and extend $\phi_i$ to the current tree $T$ by taking an embedding of $T$ into $G$ with each $w \in T$ mapped to $S_w = V(G) \setminus R_w$, where
$$
    R_w = \{ v \in V(G) \colon \exists S_{\mathbf{v}} = (v_1, \ldots, v_{i-1}, v) \text{ such that } |S_{\mathbf{v}}| = n^{(d - i)/d} \text{ and } w \in S_{\mathbf{v}}^- \},
$$
where $S_{\mathbf{v}}$ is defined with respect to the current partial homomorphism $\phi_i$ and $S_{\mathbf{v}}^- = (v_1, \ldots, v_{i-1})$. By taking $C$ to be sufficiently large we have $|R_w| < \eps m$, thus we can apply Lemma \ref{lemma:tree} to find a desired embedding.
\end{proof}

Finally, we are in position to say something about the difference between 
the proofs of Theorem \ref{thm:rational} and Theorem \ref{thm:integer}. In Theorem \ref{thm:integer} we are able to `cut' forests in such a way that each tree is of size $o(v(G))$, where $G$ is an expander. This greatly helps us with planning how to embed the vertices such that the homomorphisms are as dispersed as possible: Embed one tree, revise forbidden subsets for images of some vertices, embed the next tree, revise, and so on. The fact that for each next tree we can freely choose where the root is embedded makes it possible to implement this strategy. In contrast, in the proof of Theorem \ref{thm:rational} we cannot `cut' forests in this way: We would need to remove $O(n^{1 - 1/a})$ vertices, resulting in $O(n^{2 - 1/a})$ edges in $\Gamma^+$ which is way too much. Instead we need to find a homomorphism of the whole $H_i$ at once, and consequently we cannot do the planning one tree at a time. We resort to randomness, and drift away from the optimal bound in order to beat certain union bound. It would be interesting to improve this and resolve Conjecture \ref{conj:bounded_deg_density}.

\section{Concluding remarks and open problems}

\begin{itemize}
\item It is possible to decrease the number of vertices in the constructions
in all three main theorems to $(1+\eps)n$,
increasing the number of edges by a factor of $c(\eps)$. This can be done
following the construction in Theorem~5 of~\cite{alon2001near} 
which is based
on an appropriate concentrator (unbalanced expander).
\item
The proofs of all theorems provide efficient (deterministic or
randomized) algorithms for
embedding a given input graph $H$ of the corresponding
family in the appropriate universal graph.
\item
The proof of Theorem \ref{thm:integer_unbounded} can be easily
extended to provide economical universal graphs for any family 
of graphs on $n$ vertices in which the edges of each graph
in the family can be partitioned into a given number $d$
of subgraphs from a family with strongly sublinear
separators. Indeed it need only be possible to
break each of these subgraphs into small connected components
by removing a relatively small number of vertices.
The number of edges will depend on the size of the separators.
\item
There are several natural classes of sparse graphs that are subsets
of the family of graphs with appropriate bounded 
density. 
Notable examples are graphs with a bounded
acyclic chromatic number, graphs with a bounded arboricity,
degenerate graphs and graphs with a bounded maximum degree.
Here are some brief details.

The acyclic chromatic number of a graph $H$ is the minimum integer
$k$ so that there is a proper vertex coloring of $H$ by $k$ colors
and the vertices of each cycle of $H$ receive at least $3$ distinct colors.
Equivalently this means that there is a proper vertex coloring  of
$H$ by $k$ colors
so that the induced subgraph on the union of any two color classes
is ayclic, that is, a forest. A graph $H$ is $k$-degenerate  
if every subgraph of it contains a vertex of degree at most $k$.
A graph has arboricity $k$ if its edge-set can be
partitioned into $k$ forests.

It is not difficult to check that if
the acylic chromatic number of a graph $H$ is $k$, then 
every nonempty subset $U$ of its vertices spans at most 
$(k-1)(|U|-1)$ edges. Therefore, by Edmonds' Matroid Decomposition Theorem 
\cite{edmonds65matroid} (which for the graphic matroid that is the 
one relevant here has been proved earlier by Nash-Williams \cite{NW}),
the arboricity of $H$ is at most $k-1$. If the aroboricity is
at most $k-1$ then the density $m(H)$ is also clearly at most
$k-1$. Another simple observation is that if $H$ is
$d$-degenerate, then its arboricity (and hence also its density)
is at most $d$. Finally it is obvious that if the maximum degree of
$H$ is $t$ then its density is at most $t/2$. 
It follows that the main theorems
in this paper also provide economical constructions of universal 
graphs for the families of $n$-vertex graphs in each of these
classes.
\item
The main open problem remaining is the assertion of Conjecture
\ref{conj:bounded_density}
for all rationals $d>1$. The results proved here
as well as the special cases established in
\cite{allen2023universality,alon08optimal,chung83spanning} 
indicate that it is likely to hold in full generality.
\end{itemize}

\section*{Acknowledgements}
This project began at the 2023 REU ``Combinatorics and Coding Theory
in the Tropics'' which was run by Anant Godbole and Fernando Pi\~{n}ero.
The authors are grateful to them for their impact, as well
as to other participants of the REU, especially
Mackenzie Bookamer, Sarah Capute, and Liza Ter-Saakov. We also
thank Kalina Petrova for pointing out a helpful reference.


\begin{thebibliography}{10}

\bibitem{aggarwal96optical}
A.~Aggarwal, A.~Bar-Noy, D.~Coppersmith, R.~Ramaswami, B.~Schieber, and
  M.~Sudan.
\newblock Efficient routing in optical networks.
\newblock {\em J. ACM}, 43(6):973--1001, 1996.

\bibitem{allen2023universality}
P.~Allen, J.~B{\"o}ttcher, and A.~Liebenau.
\newblock Universality for graphs of bounded degeneracy.
\newblock {\em arXiv preprint arXiv:2309.05468}, 2023.

\bibitem{alon02treecut}
N.~Alon.
\newblock Covering a hypergraph of subgraphs.
\newblock {\em Discrete Math.}, 257(2-3):249--254, 2002.

\bibitem{alon21expander}
N.~Alon.
\newblock Explicit expanders of every degree and size.
\newblock {\em Combinatorica}, 41(4):447--463, 2021.

\bibitem{alon02sparse}
N.~Alon and V.~Asodi.
\newblock Sparse universal graphs.
\newblock {\em J. Comput. Appl. Math.}, 142(1):1--11, 2002.

\bibitem{alon07sparsebounded}
N.~Alon and M.~Capalbo.
\newblock Sparse universal graphs for bounded-degree graphs.
\newblock {\em Random Struct. Algorithms}, 31(2):123--133, 2007.

\bibitem{alon08optimal}
N.~Alon and M.~Capalbo.
\newblock Optimal universal graphs with deterministic embedding.
\newblock In {\em Proceedings of the nineteenth annual ACM-SIAM symposium on discrete algorithms, SODA 2008, San Francisco, CA, January 20--22, 2008}, pages 373--378. New York, NY: Association for Computing Machinery (ACM); Philadelphia, PA: Society for Industrial {and} Applied Mathematics (SIAM), 2008.

\bibitem{alon2000universality}
N.~Alon, M.~Capalbo, Y.~Kohayakawa, V.~Rodl, A.~Rucinski, and E.~Szemer{\'e}di.
\newblock Universality and tolerance.
\newblock In {\em Proceedings 41st Annual Symposium on Foundations of Computer Science}, pages 14--21. IEEE, 2000.

\bibitem{alon2001near}
N.~Alon, M.~Capalbo, Y.~Kohayakawa, V.~R{\"o}dl, A.~Ruci{\'n}ski, and E.~Szemer{\'e}di.
\newblock Near-optimum universal graphs for graphs with bounded degrees.
\newblock In {\em International Workshop on Randomization and Approximation Techniques in Computer Science}, pages 170--180. Springer, 2001.

\bibitem{alon20distributed}
N.~Alon, E.~Mossel, and R.~Pemantle.
\newblock Distributed corruption detection in networks.
\newblock {\em Theory Comput.}, 16:23, 2020.
\newblock Id/No 1.

\bibitem{alon16book}
N.~Alon and J.~H. Spencer.
\newblock {\em The probabilistic method}.
\newblock Wiley-Intersci. Ser. Discrete Math. Optim. Hoboken, NJ: John Wiley \& Sons, 4th edition edition, 2016.

\bibitem{babai82planar}
L.~Babai, F.~R.~K. Chung, P.~Erd{\H{o}}s, R.~L. Graham, and J.~H. Spencer.
\newblock On graphs which contain all sparse graphs.
\newblock Ann. {Discrete} {Math}. 12, 21-26 (1982)., 1982.

\bibitem{beckfiala}
J.~Beck and T.~Fiala.
\newblock ``{Integer}-making'' theorems.
\newblock {\em Discrete Appl. Math.}, 3:1--8, 1981.

\bibitem{bhatt1986optimal}
S.~Bhatt, F.~Chung, T.~Leighton, and A.~Rosenberg.
\newblock Optimal simulations of tree machines.
\newblock In {\em 27th Annual Symposium on Foundations of Computer Science}, pages 274--282. IEEE, 1986.

\bibitem{bhatt84vlsi}
S.~Bhatt and T.~Leighton.
\newblock A framework for solving {VLSI} graph layout problems.
\newblock {\em J. Comput. Syst. Sci.}, 28:300--343, 1984.

\bibitem{bollobas81threshold}
B.~Bollob\'as.
\newblock Threshold functions for small subgraphs.
\newblock {\em Math. Proc. Camb. Philos. Soc.}, 90:197--206, 1981.

\bibitem{capalbo02planar}
M.~Capalbo.
\newblock Small universal graphs for bounded-degree planar graphs.
\newblock {\em Combinatorica}, 22(3):345--359, 2002.

\bibitem{capalbo99small}
M.~R. Capalbo and S.~R. Kosaraju.
\newblock Small universal graphs.
\newblock In {\em Proceedings of the 31st annual ACM symposium on theory of computing, STOC 1999. Atlanta, GA, USA, May 1--4, 1999}, pages 741--749. New York, NY: ACM, Association for Computing Machinery, 1999.

\bibitem{chung90separator}
F.~R.~K. Chung.
\newblock Separator theorems and their applications.
\newblock Paths, flows, and {VLSI}-layout, {Proc}. {Meet}., {Bonn}/{Ger}. 1988, {Algorithms} {Comb}. 9, 17-34., 1990.

\bibitem{chung78tree}
F.~R.~K. Chung and R.~L. Graham.
\newblock On graphs which contain all small trees.
\newblock {\em J. Comb. Theory, Ser. B}, 24:14--23, 1978.

\bibitem{chung83spanning}
F.~R.~K. Chung and R.~L. Graham.
\newblock On universal graphs for spanning trees.
\newblock {\em J. Lond. Math. Soc., II. Ser.}, 27:203--211, 1983.

\bibitem{chung78tree2}
F.~R.~K. Chung, R.~L. Graham, and N.~Pippenger.
\newblock On graphs which contain all small trees. {II}.
\newblock Combinatorics, {Keszthely} 1976, {Colloq}. {Math}. {Soc}. {Janos} {Bolyai} 18, 213-223 (1978)., 1978.

\bibitem{chung83storage}
F.~R.~K. Chung, A.~L. Rosenberg, and L.~Snyder.
\newblock Perfect storage representations for families of data structures.
\newblock {\em SIAM J. Algebraic Discrete Methods}, 4:548--565, 1983.

\bibitem{draganic22rolling}
N.~Dragani{\'c}, M.~Krivelevich, and R.~Nenadov.
\newblock Rolling backwards can move you forward: on embedding problems in sparse expanders.
\newblock {\em Trans. Am. Math. Soc.}, 375(7):5195--5216, 2022.

\bibitem{edmonds65matroid}
J.~Edmonds.
\newblock Minimum partition of a matroid into independent subsets.
\newblock {\em J. Res. Natl. Bur. Stand., Sect. B}, 69:67--72, 1965.

\bibitem{esperet23planar}
L.~Esperet, G.~Joret, and P.~Morin.
\newblock Sparse universal graphs for planarity.
\newblock {\em J. Lond. Math. Soc., II. Ser.}, 108(4):1333--1357, 2023.

\bibitem{feldman88nonblocking}
P.~Feldman, J.~Friedman, and N.~Pippenger.
\newblock Wide-sense nonblocking networks.
\newblock {\em SIAM J. Discrete Math.}, 1(2):158--173, 1988.

\bibitem{friedman87tree}
J.~Friedman and N.~Pippenger.
\newblock Expanding graphs contain all small trees.
\newblock {\em Combinatorica}, 7:71--76, 1987.

\bibitem{HLW}
S. Hoory, N. Linial, and A. Wigderson. 
\newblock Expander graphs and their applications. 
\newblock {\em Bull. Amer. Math.  Soc.}, 43(04):439–562.

\bibitem{kahn21hamilton}
J.~Kahn, B.~Narayanan, and J.~Park.
\newblock The threshold for the square of a {Hamilton} cycle.
\newblock {\em Proc. Am. Math. Soc.}, 149(8):3201--3208, 2021.

\bibitem{lubotzky88ramanujan}
A.~Lubotzky, R.~Phillips, and P.~Sarnak.
\newblock Ramanujan graphs.
\newblock {\em Combinatorica}, 8(3):261--277, 1988.


\bibitem{NW}
C. St. J. A. Nash-Williams.
\newblock Decomposition of finite graphs into forests. 
\newblock {\em J. London Math. Soc.}, 39 (1964), 12.

\bibitem{nenadov2016ramsey}
R.~Nenadov.
\newblock {\em Ramsey and universality properties of random graphs}.
\newblock PhD thesis, ETH Zurich, 2016.

\bibitem{rucinski86balanced}
A.~Ruci{\'n}ski and A.~Vince.
\newblock Strongly balanced graphs and random graphs.
\newblock {\em J. Graph Theory}, 10(2):251--264, 1986.

\bibitem{Welsh}
D.~J.~A.~Welsh. 
\newblock {\em Matroid Theory}. 
\newblock L. M. S. Monographs, No. 8.  
Academic Press, London - New York, 1976, xi+433 pp.

\end{thebibliography}
\end{document}